\documentclass{amsart}
\usepackage{amsmath,amsfonts,amsthm}
\usepackage{amssymb,latexsym}
\usepackage{graphics}
\usepackage[colorlinks]{hyperref}

\usepackage{amssymb}

\theoremstyle{plain}
\theoremstyle{definition}
\newtheorem{theorem}{Theorem}[section]
\newtheorem{lemma}[theorem]{Lemma}

\newtheorem{corollary}[theorem]{Corollary}

\newtheorem{convention}[theorem]{Convention}
\newtheorem{remark}[theorem]{Remark}
\theoremstyle{remark}

\numberwithin{equation}{section}
\newcommand{\SP}{\: \: \: \: \:}

\title[Eigenvalues of weighted generalized shifts]{Eigenvalues of weighted generalized shifts over direct products of vector spaces}
\author[S. Arzanesh, F. Ayatollah Zadeh Shirazi, A. Hosseini, R. Rezavand]{Safoura Arzanesh, Fatemah Ayatollah Zadeh Shirazi, \\ Arezoo Hosseini, Reza Rezavand}
\begin{document}
%%%%%%%%%%%%%%%%%%%%%%%%%%%%%%%%%%%% abstract
\begin{abstract}
In the following text for vector space $V$ over field $F$ we compute all eigenvalues of weighted generalized shift $\sigma_{\varphi,\mathfrak{w}}:V^\Gamma\to V^\Gamma$
(and generalized shift $\sigma_\varphi:V^\Gamma\to V^\Gamma$)
for nonempty set $\Gamma$, weight vector $\mathfrak{w}\in F^\Gamma$ and self--map $\varphi:\Gamma\to \Gamma$.
\end{abstract}
\maketitle
%%%%%%%%%%%%%%%%%%%%%%%%%%%%%%%%%%%% MSC
\noindent {\small {\bf 2020 Mathematics Subject Classification:}  15A18  \\
{\bf Keywords:}} Eigenvalue, Weighted composition operator, Weighted generalized shift.
%%%%%%%%%%%%%%%%%%%%%%%%%%%%%%%%%%%%
\section{Introduction}
\noindent Computing eigenvalues of linear maps is one of
favourites of a large group of mathematicians (see e.g., J. von
Neumann;s work in 1928~\cite{neumann}). On the other hand
generalized shifts has been introduced for the first time in
\cite{karami} as a generalization of one--sided and two--sided
shifts, moreover weighted generalized shifts has been introduced
as a common generalization of generalized shifts and weighted
operators, in \cite{ebrahimi}, however in brief words it is
simply just weighted composition operator~\cite{jeang} (for study
of spectrum in weighted composition operators one may note
\cite{GUNATILLAKE} too).
\\
Our main aim in the following text is to compute all eigenvalues
of a weighted generalized shift over a direct product of linear
vector spaces.
%%%%%%%%%%%%%%%%%%%%%%%%%%%%%%%%%%%%
\begin{convention}
In the following text $V$ is a linear vector space over field $F$, $\varphi:\Gamma\to\Gamma$ is an arbitrary self--map on nonempty set $\Gamma$, and $\mathfrak{w}=(\mathfrak{w}_\alpha)_{\alpha\in\Gamma}\in
F^\Gamma$ (weight vector).
So one may consider our targets of study, the following linear maps:
\begin{itemize}
\item generalized shift
\[\sigma_\varphi:V^\Gamma\to V^\Gamma\SP,\SP
(x_\alpha)_{\alpha\in\Gamma}\mapsto(x_{\varphi(\alpha)})_{\alpha\in\Gamma}\]
\item weighted generalized shift
\[\sigma_{\varphi,\mathfrak{w}}:V^\Gamma\to V^\Gamma\SP,\SP(x_\alpha)_{\alpha\in\Gamma}\mapsto(\mathfrak{w}_\alpha x_{\varphi(\alpha)})_{\alpha\in\Gamma}(=\mathfrak{w} \sigma_\varphi((x_\alpha)_{\alpha\in\Gamma}))\:.\]
\end{itemize}
For $x=(x_\alpha)_{\alpha\in\Gamma}\in V^\Gamma,\mathfrak{v}=(\mathfrak{v}_\alpha)_{\alpha\in\Gamma}\in F^\Gamma$ and nonempty subset $D$  of
$\Gamma$ let $x^D:=(x_\alpha)_{\alpha\in D},\mathfrak{v}^D:=(\mathfrak{v}_\alpha)_{\alpha\in D}$. Also let:
\[\mathfrak{Z}:=\{\alpha\in\Gamma:\mathfrak{w}_\alpha=0\}\SP,\SP\downarrow\mathfrak{Z}=\bigcup\{\varphi^{-n}(
\mathfrak{Z}):n\geq0\}\:.\]
Let's denote zero of $F$ and zero of $V$ simply by ``$0$", also  $\mathbf{0}:=(0)_{\alpha\in\Gamma}$
is zero of $V^\Gamma$. For vector space $W$ over field $F$ and linear map $T:W\to W$ suppose $Eigen(T,W)$
is the collection of all eigenvalues of $T$, i.e., all $r\in F$ such that there exists nonzero vector $x\in W$ with $T(x)=rx$.
\end{convention}
%%%%%%%%%%%%%%%%%%%%%%%%%%%%%%%%%%%%
\section{Eigenvalues of $\sigma_{\varphi,\mathfrak{w}}:V^\Gamma\to V^\Gamma$}
%%%%%%%%%%%%%%%%%%%%%%%%%%%%%%%%%%%%
\noindent In this section we compute all eigenvalues of $\sigma_{\varphi,\mathfrak{w}}:V^\Gamma\to V^\Gamma$.
For this aim let's classify points of an arbitrary set $A$ with respect to self--map $f:A\to A$. We say $a\in A$:
\begin{itemize}
\item is a wandering point of $f$, if $\{f^n(a)\}_{n\geq1}$ is a one--to--one (infinite) sequence,
\item is a quasi--periodic or non--wandering point, if there exists $n>m\geq1$ with $f^n(a)=f^m(a)$,
\item is a periodic point if there exists $n\geq1$ with $f^n(a)=a$.
\end{itemize}
We denote
\begin{itemize}
\item the collection of all wandering points of $f$ by $W(f)$,
\item the collection of all quasi--periodic points of $f$ by $Q(f)$,
\item the collection of all periodic points of $f$ by $P(f)$.
\end{itemize}
For $\alpha\in P(f)$, let
$per(\alpha)=\min\{n\geq1:f^n(\alpha)=\alpha\}$.
\\
Let $\thicksim:=\{(\alpha,\beta)\in\Gamma\times\Gamma:\exists n,m\geq1\SP\varphi^n(\alpha)=\varphi^m(\beta)\}$, then
$\thicksim$ is an equivalence relation on $\Gamma$ (for $\alpha\in\Gamma$, $\frac{\alpha}{\thicksim}=\{\beta\in\Gamma:
\alpha\thicksim\beta\}$ is the equivalence class of $\alpha$ with respect to $\thicksim$, and
$\frac{\Gamma}{\thicksim}=\{\frac{\beta}{\thicksim}:\beta\in\Gamma\}$).
\\
Henceforth consider $\mathfrak{g}\in \Gamma$ and let:
\[\begin{array}{l}
\Gamma_\mathfrak{g}:=\dfrac{\mathfrak{g}}{\thicksim} \:, \\
\varphi_\mathfrak{g}:=\varphi\restriction_{\Gamma_\mathfrak{g}}:\Gamma_\mathfrak{g}\to \Gamma_\mathfrak{g}\:, \\
\thicksim_\mathfrak{g}:=\thicksim\cap(\Gamma_\mathfrak{g}\times\Gamma_\mathfrak{g})(=\Gamma_\mathfrak{g}\times\Gamma_\mathfrak{g})\:,  \\
\mathfrak{Z}_\mathfrak{g};=\mathfrak{Z}\cap \Gamma_\mathfrak{g}(=\{\alpha\in \Gamma_\mathfrak{g}:\mathfrak{w}_\alpha=0\})\:, \\
\downarrow\mathfrak{Z}_\mathfrak{g}:=\downarrow\mathfrak{Z}\cap \Gamma_\mathfrak{g}(=
\bigcup\{\varphi_\mathfrak{g}^{-n}(\mathfrak{Z}_\mathfrak{g}):n\geq0\})\:.
\end{array}\]
%%%%%%%%%%%%%%%%%%%%%%%%%%%%%%%%%%%%
\begin{remark}\label{salam10}
We have
$ker(\sigma_{\varphi,\mathfrak{w}})=\{(x_\alpha)_{\alpha\in\Gamma}\in
V^\Gamma:\forall\alpha\in
\varphi(\Gamma\setminus\mathfrak{Z})\:\: x_\alpha=0\}$.
\\
Hence $\sigma_{\varphi,\mathfrak{w}}:V^\Gamma\to V^\Gamma$ is
one--to--one if and only if
$\Gamma=\varphi(\Gamma\setminus\mathfrak{Z})$. It is well-known
that $0\in {\rm Eigen}(\sigma_{\varphi,\mathfrak{w}},V^\Gamma)$
if and only if $\sigma_{\varphi,\mathfrak{w}}:V^\Gamma\to
V^\Gamma$ is not one--to--one. Hence $0\in {\rm
Eigen}(\sigma_{\varphi,\mathfrak{w}},V^\Gamma)$ if and only if
$\varphi(\Gamma\setminus\mathfrak{Z})\neq\Gamma$.
\end{remark}
%%%%%%%%%%%%%%%%%%%%%%%%%%%%%%%%%%%%
\begin{lemma}\label{salam20}
If $\downarrow{\mathfrak Z}=\Gamma$, then:
\[{\rm Eigen}(\sigma_{\varphi,\mathfrak{w}},V^\Gamma)\subseteq\{0\}\:.\]
Moreover in the above case by Remark~\ref{salam10}, ${\rm Eigen}(\sigma_{\varphi,\mathfrak{w}},V^\Gamma)=\varnothing$ if $\varphi(\Gamma\setminus\mathfrak{Z})=\Gamma$
and ${\rm Eigen}(\sigma_{\varphi,\mathfrak{w}},V^\Gamma)=\{0\}$ otherwise.
\end{lemma}
%%%%%%%%%%%%%%%%%%%%%%%%%%%%%%%%%%%%
\begin{proof}
Suppose $r\in {\rm Eigen}(\sigma_{\varphi,\mathfrak{w}},V^\Gamma)$, then there exists
$x=(x_\alpha)_{\alpha\in\Gamma}\in V^\Gamma$ with $x\neq\mathbf{0}$ and
$\sigma_{\varphi,\mathfrak{w}}(x)=rx$. There exists $\theta\in\Gamma$ with $x_\theta\neq0$.
By $\theta\in\Gamma=\downarrow{\mathfrak Z}$ there
exists $n\geq0$ with $\mathfrak{w}_{\varphi^n(\theta)}=0$. Hence:
\begin{eqnarray*}
\sigma_{\varphi,\mathfrak{w}}(x)=rx & \Rightarrow & \sigma_{\varphi,\mathfrak{w}}^{n+1}(x)=r^{n+1}x \\
& \Rightarrow & \mathfrak{w}_\theta\mathfrak{w}_{\varphi(\theta)}\cdots\mathfrak{w}_{\varphi^n(\theta)}
    x_{\varphi^{n+1}(\theta)}=r^{n+1}x_\theta \\
& \Rightarrow & r^{n+1} x_\theta=0 \\
& \Rightarrow & r=0\:.
\end{eqnarray*}
\end{proof}
%%%%%%%%%%%%%%%%%%%%%%%%%%%%%%%%%%%%
\begin{corollary}\label{salam25}
If $\dfrac{\mathfrak{g}}{\thicksim}\subseteq\downarrow\mathfrak{Z}$, then
${\rm Eigen}(\sigma_{\varphi\restriction_{\frac{\mathfrak{g}}{\thicksim}},\mathfrak{w}^{\frac{\mathfrak{g}}{\thicksim}}},V^{\frac{\mathfrak{g}}{\thicksim}})\setminus\{0\}=\varnothing$.
\end{corollary}
%%%%%%%%%%%%%%%%%%%%%%%%%%%%%%%%%%%
\begin{proof}
$\Gamma_\mathfrak{g}=\downarrow\mathfrak{Z}_{\mathfrak{g}}$ since
$\dfrac{\mathfrak{g}}{\thicksim}\subseteq\downarrow\mathfrak{Z}$. Now use
Lemma~\ref{salam20} for weighted generalized shift $\sigma_{\varphi_{\mathfrak g},\mathfrak{w}^{\Gamma_{\mathfrak g}}}:\Gamma_{\mathfrak g}\to\Gamma_{\mathfrak g}$.
\end{proof}
%%%%%%%%%%%%%%%%%%%%%%%%%%%%%%%%%%%%
\begin{lemma}\label{salam30}
${\rm Eigen}(\sigma_{\varphi,\mathfrak{w}},V^\Gamma)=\bigcup\{{\rm Eigen}(\sigma_{\varphi\restriction_{\frac\alpha\thicksim},\mathfrak{w}^{\frac\alpha\thicksim}},V^{\frac\alpha\thicksim}):\alpha\in\Gamma\}$.
\end{lemma}
%%%%%%%%%%%%%%%%%%%%%%%%%%%%%%%%%%%%
\begin{proof}
First suppose $r\in {\rm Eigen}(\sigma_{\varphi,\mathfrak{w}})$, then there exists
$x=(x_\alpha)_{\alpha\in\Gamma}\in V^\Gamma$ such that $x\neq\mathbf{0}$ and
$\sigma_{\varphi,\mathfrak{w}}(x)=rx$. Choose $\theta\in\Gamma$ with $x_\theta\neq0$, then we have
$x^\frac\theta\thicksim\neq\mathbf{0}^\frac\theta\thicksim$ and:
\begin{eqnarray*}
\sigma_{\varphi,\mathfrak{w}}(x)=rx & \Rightarrow & (\forall\alpha\in\Gamma\:\:\: \mathfrak{w}_\alpha x_{\varphi(\alpha)}=
    rx_\alpha) \\
& \Rightarrow & (\forall\alpha\in\frac\theta\thicksim\:\:\: \mathfrak{w}_\alpha x_{\varphi(\alpha)}=rx_\alpha) \\
& \Rightarrow & \sigma_{\varphi\restriction_{\frac\theta\thicksim},\mathfrak{w}^\frac\theta\thicksim}(x^\frac\theta\thicksim)
    =rx^\frac\theta\thicksim \\
& \mathop{\Rightarrow}\limits^{x^\frac\theta\thicksim\neq\mathbf{0}^\frac\theta\thicksim} & r\in {\rm Eigen}(\sigma_{\varphi\restriction_{\frac\theta\thicksim},\mathfrak{w}^{\frac\theta\thicksim}},V^{\frac\theta\thicksim})
\end{eqnarray*}
hence ${\rm Eigen}(\sigma_{\varphi,\mathfrak{w}},V^\Gamma)\subseteq\bigcup\{{\rm Eigen}(\sigma_{\varphi\restriction_{\frac\alpha\thicksim},\mathfrak{w}^{\frac\alpha\thicksim}},V^{\frac\alpha\thicksim}):\alpha\in\Gamma\}$.
\\
Now, for $\beta\in\Gamma$ suppose $s\in {\rm Eigen}(\sigma_{\varphi\restriction_{\frac\beta\thicksim},\mathfrak{w}^{\frac\beta\thicksim}},V^{\frac\beta\thicksim})$, then there exists $z=(z_\alpha)_{\alpha\in\frac\beta\thicksim}$
with $z\neq\mathbf{0}^\frac\beta\thicksim$ and $\sigma_{\varphi\restriction_{\frac\theta\thicksim},\mathfrak{w}^{\frac\theta\thicksim}}(z)=sz$. Let:
\[y_\alpha:=\left\{\begin{array}{lc} z_\alpha\:, & \alpha\in\frac\theta\thicksim\:, \\ 0\:, & otherwise\:, \end{array}\right.\]
then for $y:=(y_\alpha)_{\alpha\in\Gamma}$ we have $y\neq\mathbf{0}$ and $\sigma_{\varphi,\mathfrak{w}}(y)=sy$
which leads to $s\in {\rm Eigen}(\sigma_{\varphi,\mathfrak{w}},V^\Gamma)$ and completes the proof.
\end{proof}
%%%%%%%%%%%%%%%%%%%%%%%%%%%%%%%%%%%%
\begin{lemma}\label{salam40}
If $W(\varphi)\setminus \downarrow{\mathfrak Z}\neq\varnothing$, then:
\[F\setminus\{0\}\subseteq {\rm Eigen}(\sigma_{\varphi,\mathfrak{w}},V^\Gamma)\:.\]
Moreover in the above case by Remark~\ref{salam10}, ${\rm Eigen}(\sigma_{\varphi,\mathfrak{w}},V^\Gamma)=F\setminus\{0\}$ if $\varphi(\Gamma\setminus\mathfrak{Z})=\Gamma$ and
${\rm Eigen}(\sigma_{\varphi,\mathfrak{w}},V^\Gamma)=F$ otherwise.
\end{lemma}
%%%%%%%%%%%%%%%%%%%%%%%%%%%%%%%%%%%%
\begin{proof}
Choose $r\in F\setminus\{0\}$ and
$\theta\in W(\varphi)\setminus \downarrow{\mathfrak Z}$ and $v\in V\setminus\{0\}$. Suppose $\alpha\in
\frac{\theta}{\thicksim}$, then there exist $p,q\geq 1$ with $\varphi^p(\alpha)=\varphi^q(\theta)$. Let
\[x_\alpha:=r^{q-p}(\mathfrak{w}_\alpha\mathfrak{w}_{\varphi(\alpha)}\cdots\mathfrak{w}_{\varphi^p(\alpha)})
(\mathfrak{w}_\theta\mathfrak{w}_{\varphi(\theta)}\cdots\mathfrak{w}_{\varphi^q(\theta)})^{-1} v\:.\]
Note that $x_\alpha$ does not depend on chosen $p,q$, since for $p_1,q_1\geq1$ with
$\varphi^{p_1}(\alpha)=\varphi^{q_1}(\theta)$, we may suppose $p_1\geq p$, then:
\[\varphi^{q_1}(\theta)=\varphi^{p_1}(\alpha)=\varphi^{p_1-p}(\varphi^p(\alpha))=\varphi^{p_1-p}(\varphi^q(\theta))
=\varphi^{p_1-p+q}(\theta)\:.\]
By $\theta\in W(\varphi)$ and $\varphi^{q_1}(\theta)=\varphi^{p_1-p+q}(\theta)$ we conclude
$q_1=p_1-p+q$ and $p-q=p_1-q_1$ in particular $q_1\geq q$, moreover:
\\
$r^{q_1-p_1}(\mathfrak{w}_\alpha\mathfrak{w}_{\varphi(\alpha)}\cdots\mathfrak{w}_{\varphi^{p_1}(\alpha)})
(\mathfrak{w}_\theta\mathfrak{w}_{\varphi(\theta)}\cdots\mathfrak{w}_{\varphi^{q_1}(\theta)})^{-1}$
\begin{eqnarray*}
& = & r^{q-p}(\mathfrak{w}_\alpha\mathfrak{w}_{\varphi(\alpha)}\cdots\mathfrak{w}_{\varphi^{p_1}(\alpha)})
(\mathfrak{w}_\theta\mathfrak{w}_{\varphi(\theta)}\cdots\mathfrak{w}_{\varphi^{q_1}(\theta)})^{-1} \\
& = & r^{q-p}(\mathfrak{w}_\alpha\mathfrak{w}_{\varphi(\alpha)}\cdots\mathfrak{w}_{\varphi^{p}(\alpha)})
(\mathfrak{w}_{\varphi^{p+1}(\alpha)}\cdots\mathfrak{w}_{\varphi^{p_1}(\alpha)})
\\ && \SP\SP\SP\SP\SP\SP\SP
(\mathfrak{w}_{\varphi^{q+1}(\theta)}\cdots \mathfrak{w}_{\varphi^{q_1}(\theta)})^{-1}
(\mathfrak{w}_\theta\mathfrak{w}_{\varphi(\theta)}\cdots\mathfrak{w}_{\varphi^{q}(\theta)})^{-1} \\
& = & r^{q-p}(\mathfrak{w}_\alpha\mathfrak{w}_{\varphi(\alpha)}\cdots\mathfrak{w}_{\varphi^{p}(\alpha)})
(\mathfrak{w}_{\varphi^{q+1}(\theta)}\cdots\mathfrak{w}_{\varphi^{p_1-p+q}(\theta)})
\\ && \SP\SP\SP\SP\SP\SP\SP
(\mathfrak{w}_{\varphi^{q+1}(\theta)}\cdots \mathfrak{w}_{\varphi^{q_1}(\theta)})^{-1}
(\mathfrak{w}_\theta\mathfrak{w}_{\varphi(\theta)}\cdots\mathfrak{w}_{\varphi^{q}(\theta)})^{-1} \\
& = & r^{q-p}(\mathfrak{w}_\alpha\mathfrak{w}_{\varphi(\alpha)}\cdots\mathfrak{w}_{\varphi^{p}(\alpha)})
(\mathfrak{w}_{\varphi^{q+1}(\theta)}\cdots\mathfrak{w}_{\varphi^{q_1}(\theta)})
\\ && \SP\SP\SP\SP\SP\SP\SP
(\mathfrak{w}_{\varphi^{q+1}(\theta)}\cdots \mathfrak{w}_{\varphi^{q_1}(\theta)})^{-1}
(\mathfrak{w}_\theta\mathfrak{w}_{\varphi(\theta)}\cdots\mathfrak{w}_{\varphi^{q}(\theta)})^{-1} \\
& = & r^{q-p}(\mathfrak{w}_\alpha\mathfrak{w}_{\varphi(\alpha)}\cdots\mathfrak{w}_{\varphi^{p}(\alpha)})
(\mathfrak{w}_\theta\mathfrak{w}_{\varphi(\theta)}\cdots\mathfrak{w}_{\varphi^{q}(\theta)})^{-1} \:.
\end{eqnarray*}
Moreover for $\beta\in\frac\theta\thicksim$, if $s,t\geq1$ and $\varphi^s(\beta)=\varphi^t(\theta)$,
then $\varphi^s(\varphi(\beta))=\varphi^{s+1}(\beta)=\varphi^{t+1}(\theta)$ and:
\[x_\beta=r^{t-s}(\mathfrak{w}_\beta\mathfrak{w}_{\varphi(\beta)}\cdots\mathfrak{w}_{\varphi^s(\beta)})
(\mathfrak{w}_\theta\mathfrak{w}_{\varphi(\theta)}\cdots\mathfrak{w}_{\varphi^t(\theta)})^{-1} v\:,\]
\[x_{\varphi(\beta)}=r^{t+1-s}(\mathfrak{w}_{\varphi(\beta)}\mathfrak{w}_{\varphi^2(\beta)}\cdots\mathfrak{w}_{\varphi^{s+1}(\beta)})
(\mathfrak{w}_\theta\mathfrak{w}_{\varphi(\theta)}\cdots\mathfrak{w}_{\varphi^{t+1}(\theta)})^{-1} v\:,\]
thus:
\[\mathfrak{w}_\beta x_{\varphi(\beta)}=r\mathfrak{w}_{\varphi^{s+1}(\beta)}\mathfrak{w}_{\varphi^{t+1}(\theta)}^{-1}x_\beta
=r\mathfrak{w}_{\varphi^{t+1}(\theta)}\mathfrak{w}_{\varphi^{t+1}(\theta)}^{-1}x_\beta=rx_\beta\]
which shows:
\[\sigma_{\varphi\restriction_{\frac\theta\thicksim},\mathfrak{w}^\frac\theta\thicksim}((x_\beta)_{\beta\in\frac\theta\thicksim})
=(\mathfrak{w}_\beta x_{\varphi(\beta)})_{\beta\in\frac\theta\thicksim}=r(x_\beta)_{\beta\in\frac\theta\thicksim}
\]
and $r\in {\rm Eigen}(\sigma_{\varphi\restriction_{\frac\theta\thicksim},\mathfrak{w}^\frac\theta\thicksim},V^\frac\theta\thicksim)$.
By Lemma~\ref{salam30},
 $r\in {\rm Eigen}(\sigma_{\varphi,\mathfrak{w}},V^\Gamma)$
\end{proof}
%%%%%%%%%%%%%%%%%%%%%%%%%%%%%%%%%%%%
\begin{remark}[{\cite[Remark 2.6]{soleimani}}]\label{salam50}
For each $\alpha\in\Gamma$ we have:
\begin{itemize}
\item $\alpha\in W(\varphi)$ if and only if $\frac\alpha\thicksim\subseteq W(\varphi)$,
\item $\alpha\in Q(\varphi)$ if and only if $\frac\alpha\thicksim\subseteq Q(\varphi)$.
\end{itemize}
In particular, if $\thicksim=\Gamma\times\Gamma$, then:
\begin{itemize}
\item $W(\varphi)\neq\varnothing$ if and only if $W(\varphi)=\Gamma$,
\item  $Q(\varphi)\neq\varnothing$ if and only if $Q(\varphi)=\Gamma$,
\end{itemize}
\end{remark}
%%%%%%%%%%%%%%%%%%%%%%%%%%%%%%%%%%%%
%%%%%%%%%%%%%%%%%%%%%%%%%%%%%%%%%%%%
\begin{lemma}\label{salam60}
Let $W(\varphi)\setminus \downarrow{\mathfrak Z}=\varnothing$,
$\Gamma\setminus \downarrow{\mathfrak Z}\neq\varnothing$ and $\thicksim=\Gamma\times\Gamma$, then:
\\
1. $\varnothing \neq P(\varphi)\subseteq \Gamma\setminus \downarrow{\mathfrak Z}\subseteq Q(\varphi)$,
\\
2. for all $\theta\in P(\varphi)$ we have $\{\varphi^n(\theta):n\geq0\}=P(\varphi)$,
\\
3. suppose $\theta\in P(\varphi)$, then:
{\small\[{\rm Eigen}(\sigma_{\varphi,\mathfrak{w}},V^\Gamma)=\left\{\begin{array}{lc}
 \{r\in F\setminus\{0\}: \mathfrak{w}_\theta\cdots\mathfrak{w}_{\varphi^{per(\theta)-1}(\theta)}=r^{per(\theta)}\}\:,
& \varphi(\Gamma\setminus\mathfrak{Z})=\Gamma\:, \\
\{r\in F\setminus\{0\}: \mathfrak{w}_\theta\cdots\mathfrak{w}_{\varphi^{per(\theta)-1}(\theta)}=r^{per(\theta)}\}\cup\{0\}\:,
& otherwise\:. \end{array}\right.\]}
\end{lemma}
%%%%%%%%%%%%%%%%%%%%%%%%%%%%%%%%%%%%
\begin{proof} Note that $P(\varphi)\subseteq Q(\varphi)\subseteq \Gamma=W(\varphi)\cup Q(\varphi)$
(one may consider Remark~\ref{salam50}).
\\
{\bf 1)} Since $W(\varphi)\setminus \downarrow{\mathfrak Z}=\varnothing$ we have
$\Gamma\setminus\downarrow{\mathfrak Z}=Q(\varphi)\setminus\downarrow{\mathfrak Z}\subseteq Q(\varphi)$.
\\
Choose $\kappa\in Q(\varphi)\setminus\downarrow{\mathfrak Z}=\Gamma\setminus\downarrow{\mathfrak Z}\neq
\varnothing$, then for each $n\geq1$ we have $\varphi^n(\kappa)\in \Gamma\setminus \downarrow{\mathfrak Z}$.
Since $\kappa\in Q(\varphi)$, there exists $p\geq1$ such that $\varphi^p(\kappa)\in P(\varphi)$, in particular
$P(\varphi)\neq\varnothing$.
\\
Suppose $\theta\in P(\varphi)$, by
$\thicksim=\Gamma\times\Gamma$ we have $\theta\thicksim\kappa$ and there exist $s\geq 1$, $t\in\{0,1,\ldots,per(\theta)-1\}$
with
$\varphi^s(\kappa)=\varphi^t(\theta)$, hence $\theta=\varphi^{per(\theta)}(\theta)=\varphi^{s+ per(\theta)-t}(\kappa)
\in \Gamma\setminus \downarrow{\mathfrak Z}$ which leads to $P(\varphi)\subseteq
\Gamma\setminus \downarrow{\mathfrak Z}$.
\\
{\bf 2)} Consider $\theta,\lambda\in P(\varphi)$, it's evident that $\{\varphi^n(\theta):n\geq0\}\subseteq P(\varphi)$. Since
$\theta\thicksim\lambda$, there exist $s\in\{0,\ldots,per(\lambda)-1\},t\geq1$ with $\varphi^t(\theta)=\varphi^s(\lambda)$. Thus
$\lambda=\varphi^{per(\lambda)}(\lambda)=\varphi^{t+per(\lambda)-s}(\theta)\in \{\varphi^n(\theta):n\geq0\}$. Therefore
$P(\varphi)\subseteq \{\varphi^n(\theta):n\geq0\}$.
\\
{\bf 3)} First suppose $r\in {\rm Eigen}(\sigma_{\varphi,\mathfrak{w}},V^\Gamma)\setminus\{0\}$ and $\theta\in P(\varphi)$, then
there exists $x=(x_\alpha)_{\alpha\in\Gamma}\in V^\Gamma$ such that $x\neq\mathbf{0}$ and $\sigma_{\varphi,
\mathfrak{w}}(x)=rx$.  Choose $\beta\in\Gamma$ with $x_\beta\neq0$, by $\beta\thicksim\theta$ there exist
$s\geq1$, $t\in\{0,\ldots,per(\theta)-1\}$ such that $\varphi^s(\beta)=\varphi^t(\theta)$ hence $\varphi^{s+per(\theta)-t}(\beta)=\varphi^{per(\theta)}(\theta)=\theta$ and:
\begin{eqnarray*}
\sigma_{\varphi,\mathfrak{w}}(x)=rx & \Rightarrow & \sigma_{\varphi,\mathfrak{w}}^{s+per(\theta)-t}(x)=r^{s+per(\theta)-t}x \\
& \Rightarrow & \mathfrak{w}_\beta\cdots\mathfrak{w}_{\varphi^{s+per(\theta)-t-1}(\beta)}
    x_{\varphi^{s+per(\theta)-t}(\beta)}=r^{s+per(\theta)-t}x_\beta \\
& \Rightarrow & \mathfrak{w}_\beta\cdots\mathfrak{w}_{\varphi^{s+per(\theta)-t-1}(\beta)}
    x_\theta=r^{s+per(\theta)-t}x_\beta \neq 0 \\
& \Rightarrow & x_\theta\neq0\SP\SP\SP(\divideontimes)
\end{eqnarray*}
so:
\begin{eqnarray*}
\sigma_{\varphi,\mathfrak{w}}(x)=rx & \Rightarrow & \sigma_{\varphi,\mathfrak{w}}^{per(\theta)}(x)=r^{per(\theta)}x \\
& \Rightarrow & \mathfrak{w}_\theta\cdots\mathfrak{w}_{\varphi^{per(\theta)-1}(\theta)}
    x_{\varphi^{per(\theta)}(\theta)}=r^{per(\theta)}x_\theta \\
& \Rightarrow & \mathfrak{w}_\theta\cdots\mathfrak{w}_{\varphi^{per(\theta)-1}(\theta)} x_\theta=r^{per(\theta)}x_\theta \\
& \mathop{\Rightarrow}\limits^{(\divideontimes)} & \mathfrak{w}_\theta\cdots\mathfrak{w}_{\varphi^{per(\theta)-1}(\theta)}=r^{per(\theta)}\:.
\end{eqnarray*}
Now consider $\rho\in F\setminus\{0\}$ with $\mathfrak{w}_\theta\cdots\mathfrak{w}_{\varphi^{per(\theta)-1}(\theta)}=\rho^{per(\theta)}$ and $v\in V\setminus\{0\}$. For $\alpha\in\Gamma$ there exist
$s\geq1$, $t\in\{0,\ldots,per(\theta)-1\}$ such that $\varphi^s(\alpha)=\varphi^t(\theta)$. Let:
\[y_\alpha:=\rho^{-(s+per(\theta)-t)}\mathfrak{w}_\alpha\cdots\mathfrak{w}_{\varphi^{s+per(\theta)-t-1}(\alpha)}v\:.\]
Note that $y_\alpha$ does not depend on chosen $s\geq1$, $t\in\{0,\ldots,per(\theta)-1\}$, for this aim consider $s_1\geq1$,
$t_1\in\{0,\ldots,per(\theta)-1\}$ with
$\varphi^{s_1}(\alpha)=\varphi^{t_1}(\theta)$. One may suppose $s_1\geq s$. Then
$\varphi^{t_1}(\theta)=\varphi^{s_1}(\alpha)=\varphi^{s_1-s}(\varphi^s(\alpha))=\varphi^{s_1-s+t}(\theta)$
and there exists $k\geq0$ such that $s_1-s+t=t_1+k\:per(\theta)$. Hence:

$\rho^{-(s_1+per(\theta)-t_1)}\mathfrak{w}_\alpha\cdots\mathfrak{w}_{\varphi^{s_1+ per(\theta)-t_1-1}(\alpha)} $
\begin{eqnarray*}
& = &\rho^{-(s_1+per(\theta)-t_1)}\mathfrak{w}_\alpha\cdots\mathfrak{w}_{\varphi^{s+ per(\theta)-t+k\: per(\theta) -1}(\alpha)}
\\ 
& = &
\rho^{-(s_1+per(\theta)-t_1)}(\mathfrak{w}_\alpha\cdots\mathfrak{w}_{\varphi^{s+per(\theta)-t-1}(\alpha)} )
\\ &&
(\mathfrak{w}_{\varphi^{s+per(\theta)-t}(\alpha)}\cdots \mathfrak{w}_{\varphi^{s+per(\theta)-t+per(\theta)-1}(\alpha)})
\\ &&
(\mathfrak{w}_{\varphi^{s+per(\theta)-t+per(\theta)}(\alpha)}\cdots \mathfrak{w}_{\varphi^{s+per(\theta)-t+2\:per(\theta)-1}(\alpha)})
\\ && \vdots \\ & &
(\mathfrak{w}_{\varphi^{s+per(\theta)-t+(k-1)\:per(\theta)}(\alpha)}\cdots \mathfrak{w}_{\varphi^{s+per(\theta)-t+k\:per(\theta)-1}(\alpha)})
\\ & \mathop{=}\limits^{\varphi^s(\alpha)=\varphi^t(\theta)} &
\rho^{-(s_1+per(\theta)-t_1)}(\mathfrak{w}_\alpha\cdots\mathfrak{w}_{\varphi^{s+per(\theta)-t-1}(\alpha)} )
\\ &&
(\mathfrak{w}_{\varphi^{per(\theta)}(\theta)}\cdots \mathfrak{w}_{\varphi^{per(\theta)+per(\theta)-1}(\theta)})
\\ &&
(\mathfrak{w}_{\varphi^{per(\theta)+per(\theta)}(\theta)}\cdots \mathfrak{w}_{\varphi^{per(\theta)+2\:per(\theta)-1}(\theta)})
\\ && \vdots \\ & &
(\mathfrak{w}_{\varphi^{per(\theta)+(k-1)\:per(\theta)}(\theta)}\cdots \mathfrak{w}_{\varphi^{per(\theta)+k\:per(\theta)-1}(\theta)})
\\ & = &\rho^{-(s_1+per(\theta)-t_1)}(\mathfrak{w}_\alpha\cdots\mathfrak{w}_{\varphi^{s+per(\theta)-t-1}(\alpha)} )
(\mathfrak{w}_\theta\cdots\mathfrak{w}_{\varphi^{per(\theta)-1}(\theta)})^k \\
& = & \rho^{-(s_1+per(\theta)-t_1)+k\:per(\theta)}(\mathfrak{w}_\alpha\cdots\mathfrak{w}_{\varphi^{s+per(\theta)-t-1}(\alpha)} ) \\
& = & \rho^{-(s+per(\theta)-t)}\mathfrak{w}_\alpha\cdots\mathfrak{w}_{\varphi^{s+ per(\theta)-t-1}(\alpha)} \:.
\end{eqnarray*}
For $\beta\in\Gamma$ and $p\geq2,q\in\{0,\ldots,per(\theta)-1\}$ with $\varphi^p(\beta)=\varphi^q(\theta)$ we have
$\varphi^{p-1}(\varphi(\beta))=\varphi^{q}(\theta)$ and:
\[y_\beta:=\rho^{-(p+per(\theta)-q)}\mathfrak{w}_\beta\cdots\mathfrak{w}_{\varphi^{p+per(\theta)-q-1}(\beta)}v\:,\]
and
\begin{eqnarray*}
y_{\varphi(\beta)} & = & \rho^{-(p-1+per(\theta)-q)}\mathfrak{w}_{\varphi(\beta)}\cdots\mathfrak{w}_{\varphi^{p-1+per(\theta)-q-1}(\varphi(\beta))}v
\\ & = &
\rho^{-(p-1+per(\theta)-q)}\mathfrak{w}_{\varphi(\beta)}\cdots\mathfrak{w}_{\varphi^{p+per(\theta)-q-1}(\beta)}v \\
& = & \rho \rho^{-(p+per(\theta)-q)}\mathfrak{w}_{\varphi(\beta)}\cdots\mathfrak{w}_{\varphi^{p+per(\theta)-q-1}(\beta)}v
\end{eqnarray*}
Thus $\mathfrak{w}_\beta y_{\varphi(\beta)}=\rho y_\beta$, which shows
$\sigma_{\varphi,\mathfrak{w}}((y_\alpha)_{\alpha\in\Gamma})=\rho(y_\alpha)_{\alpha\in\Gamma}$
and $\rho\in {\rm Eigen}(\sigma_{\varphi,\mathfrak{w}},V^\Gamma)$.
\end{proof}
%%%%%%%%%%%%%%%%%%%%%%%%%%%%%%%%%%%%
\begin{corollary}\label{salam65}
If $\mathfrak{g}\in Q(\varphi)\setminus \downarrow{\mathfrak Z}$,
then
$\varnothing \neq P(\varphi_\mathfrak{g})\subseteq \Gamma_\mathfrak{g}\setminus \downarrow{\mathfrak Z}_\mathfrak{g}\subseteq Q(\varphi_\mathfrak{g})$,
and ${\rm Eigen}(\sigma_{\varphi\restriction_{\frac{\mathfrak{g}}{\thicksim}},\mathfrak{w}^{\frac{\mathfrak{g}}{\thicksim}}},V^{\frac{\mathfrak{g}}{\thicksim}})\setminus\{0\}=\{r\in F\setminus\{0\}: \mathfrak{w}_\theta\cdots\mathfrak{w}_{\varphi^{per(\theta)-1}(\theta)}=r^{per(\theta)}\}$ for all $\theta\in P(\varphi_\mathfrak{g})$.
\end{corollary}
%%%%%%%%%%%%%%%%%%%%%%%%%%%%%%%%%%%%
\begin{proof}
Since $\mathfrak{g}\in Q(\varphi)$, so $\Gamma_\mathfrak{g}=\dfrac{\mathfrak g}{\thicksim}\subseteq Q(\varphi)$. Thus
$W(\varphi_\mathfrak{g})\setminus \downarrow{\mathfrak Z}_\mathfrak{g}=\varnothing\setminus \downarrow{\mathfrak Z}=\varnothing$ and
$\mathfrak{g}\in \Gamma_\mathfrak{g}\setminus \downarrow{\mathfrak Z_\mathfrak{g}}\neq\varnothing$ (note that $\thicksim_\mathfrak{g}=\Gamma_\mathfrak{g}\times\Gamma_\mathfrak{g}$). Use Lemma~\ref{salam60} to complete the proof.
\end{proof}
%%%%%%%%%%%%%%%%%%%%%%%%%%%%%%%%%%%%
\begin{theorem}[Eigenvalues of weighted generalized shifts]\label{main}
For
\[\mathsf{M}:=\{r\in F\setminus\{0\}: \exists\theta\in P(\varphi)\setminus\downarrow{\mathfrak{Z}}\mathop{\prod}\limits_{0\leq i<per(\theta)} \mathfrak{w}_{\varphi^{i}(\theta)}=r^{per(\theta)}\}\]
we have:
\[{\rm Eigen}(\sigma_{\varphi,\mathfrak{w}},V^\Gamma)=\left\{\begin{array}{lc}
F\setminus\{0\} \:, &  W(\varphi)\not\subseteq \downarrow{\mathfrak Z} ,\Gamma=\varphi(\Gamma\setminus\mathfrak{Z})\:, \\
&\\
F\:, \: & W(\varphi)\not\subseteq\downarrow{\mathfrak Z} ,\Gamma\neq\varphi(\Gamma\setminus\mathfrak{Z})\:, \\
& \\
\mathsf{M}\:, & W(\varphi)\subseteq \downarrow{\mathfrak Z}, \Gamma=\varphi(\Gamma\setminus\mathfrak{Z})\:, \\
&\\
\mathsf{M}\cup\{0\}\:, & W(\varphi)\subseteq\downarrow{\mathfrak Z},\Gamma\neq\varphi(\Gamma\setminus\mathfrak{Z})\:.
\end{array}\right.\]
\end{theorem}
%%%%%%%%%%%%%%%%%%%%%%%%%%%%%%%%%%%%
\begin{proof}
First suppose $W(\varphi)\not\subseteq \downarrow{\mathfrak Z}$, then by Lemma~\ref{salam40},
${\rm Eigen}(\sigma_{\varphi,\mathfrak{w}},V^\Gamma)=F\setminus\{0\}$ if
$\Gamma=\varphi(\Gamma\setminus\mathfrak{Z})$ and ${\rm Eigen}(\sigma_{\varphi,\mathfrak{w}},V^\Gamma)=F$ otherwise.
\\
Now suppose $W(\varphi)\subseteq \downarrow{\mathfrak Z}$. For $\theta\in\Gamma$ we
have the following cases:
\\
$\bullet$ $\theta\in W(\varphi)\subseteq \downarrow{\mathfrak Z}$. In this case by Corollary~\ref{salam25},
${\rm Eigen}(\sigma_{\varphi\restriction_{\frac{\theta}{\thicksim}},\mathfrak{w}^{\frac{\theta}{\thicksim}}},V^{\frac{\theta}{\thicksim}})\setminus\{0\}=\varnothing$.
\\
$\bullet$ $\theta\in Q(\varphi)$ and $\frac\theta\thicksim\subseteq\downarrow\mathfrak{Z}$. Again in  this case by Corollary~\ref{salam25},
${\rm Eigen}(\sigma_{\varphi\restriction_{\frac{\theta}{\thicksim}},\mathfrak{w}^{\frac{\theta}{\thicksim}}},V^{\frac{\theta}{\thicksim}})\setminus\{0\}=\varnothing$.
\\
$\bullet$ $\theta\in Q(\varphi)$ and $\frac\theta\thicksim\setminus\downarrow\mathfrak{Z}\neq\varnothing$. In  this case by Corollary~\ref{salam65}, $P(\varphi)\cap\frac\theta\thicksim\cap\downarrow\mathfrak{Z}=\varnothing$ and
${\rm Eigen}(\sigma_{\varphi\restriction_{\frac{\theta}{\thicksim}},\mathfrak{w}^{\frac{\theta}{\thicksim}}},V^{\frac{\theta}{\thicksim}})\setminus\{0\}=\{r\in F\setminus\{0\}: \mathfrak{w}_\theta\cdots\mathfrak{w}_{\varphi^{per(\theta)-1}(\theta)}=r^{per(\theta)}\}$ for each
$\theta\in P(\varphi)\cap\frac\theta\thicksim$
\\
On the other hand by Corollary~\ref{salam65}, for $\theta\in Q(\varphi)$ we have
$\frac\theta\thicksim\not\subseteq\downarrow\mathfrak{Z}$ if and only if $P(\varphi)\cap \frac\theta\thicksim\not\subseteq\downarrow\mathfrak{Z}$. Thus by Lemma~\ref{salam30} and above cases:

${\rm Eigen}(\sigma_{\varphi,\mathfrak{w}},V^\Gamma)\setminus\{0\}$
\begin{eqnarray*}
&=&\bigcup\{{\rm Eigen}(\sigma_{\varphi\restriction_{\frac\alpha\thicksim},\mathfrak{w}^{\frac\alpha\thicksim}},V^{\frac\alpha\thicksim}):\alpha\in\Gamma\}\setminus\{0\} \\
&=&\bigcup\{{\rm Eigen}(\sigma_{\varphi\restriction_{\frac\alpha\thicksim},\mathfrak{w}^{\frac\alpha\thicksim}},V^{\frac\alpha\thicksim}):\alpha\in Q(\varphi),
\frac\alpha\thicksim\setminus\downarrow\mathfrak{Z}\neq\varnothing\}\setminus\{0\} \\
&=&\bigcup\{{\rm Eigen}(\sigma_{\varphi\restriction_{\frac\alpha\thicksim},\mathfrak{w}^{\frac\alpha\thicksim}},V^{\frac\alpha\thicksim}):\alpha\in P(\varphi),
\frac\alpha\thicksim\setminus\downarrow\mathfrak{Z}\neq\varnothing\}\setminus\{0\}\\
&=&\bigcup\{{\rm Eigen}(\sigma_{\varphi\restriction_{\frac\alpha\thicksim},\mathfrak{w}^{\frac\alpha\thicksim}},V^{\frac\alpha\thicksim}):\alpha\in P(\varphi)\setminus\downarrow\mathfrak{Z}\}\setminus\{0\} \\
& = & \bigcup\{
\{r\in F\setminus\{0\}: \mathfrak{w}_\alpha\cdots\mathfrak{w}_{\varphi^{per(\alpha)-1}(\alpha)}=r^{per(\alpha)}\}:\alpha\in P(\varphi)\setminus\downarrow\mathfrak{Z}\} \\
& = &
\{r\in F\setminus\{0\}: \exists \alpha\in P(\varphi)\setminus\downarrow\mathfrak{Z}\:\:\:\: \mathfrak{w}_\alpha\cdots\mathfrak{w}_{\varphi^{per(\alpha)-1}(\alpha)}=r^{per(\alpha)}\} \\
& = & \mathsf{M}
\end{eqnarray*}
Use  Remark~\ref{salam10} to complete the proof.
\end{proof}
%%%%%%%%%%%%%%%%%%%%%%%%%%%%%%%%%%%%
\begin{remark}%\label{salam70}
If $\sigma_{\varphi,\mathfrak{w}}:V^\Gamma\to V^\Gamma$ is isomorphism, then it is
onto and $\mathfrak{Z}=\downarrow\mathfrak{Z}=\varnothing$. So by
 Theorem~\ref{main} and Remark~\ref{salam10}:
{\small \[{\rm Eigen}(\sigma_{\varphi,\mathfrak{w}},V^\Gamma)=\left\{\begin{array}{lc} 
F\setminus\{0\} \:, &  W(\varphi)\neq\varnothing \:, \\
& \\
\{r\in F\setminus\{0\}: \exists\theta\in P(\varphi)\mathop{\prod}\limits_{0\leq i<per(\theta)} \mathfrak{w}_{\varphi^{i}(\theta)}=r^{per(\theta)}\}\:, \: & otherwise\:. 
\end{array}\right.\]}
\end{remark}
%%%%%%%%%%%%%%%%%%%%%%%%%%%%%%%%%%%%
%%%%%%%%%%%%%%%%%%%%%%%%%%%%%%%%%%%%
\begin{corollary}[Eigenvalues of generalized shifts]%\label{salam80}
We have:
{\small\[{\rm Eigen}(\sigma_\varphi,V^\Gamma)=\left\{\begin{array}{lc}
\{r\in F\setminus\{0\}: \exists\theta\in P(\varphi)\:\: o(r)|per(\theta)\}\:, & W(\varphi)=\varnothing, \Gamma=\varphi(\Gamma)\:, \\
& \\
\{r\in F\setminus\{0\}: \exists\theta\in P(\varphi)\:\: o(r)|per(\theta)\}\cup\{0\}\:, & W(\varphi)=\varnothing, \Gamma\neq\varphi(\Gamma)\:, \\
& \\
F\setminus\{0\} \:, &  W(\varphi)\neq\varnothing ,\Gamma=\varphi(\Gamma)\:, \\
& \\
F\:, \: & W(\varphi)\neq\varnothing ,\Gamma\neq\varphi(\Gamma)\:.

\end{array}\right.\]}
where for $r\in F\setminus\{0\}$, $o(r)$ denotes the order of $r$ in multiplicative group $F\setminus\{0\}$.
\end{corollary}
%%%%%%%%%%%%%%%%%%%%%%%%%%%%%%%%%%%%
\begin{proof}
Use Theorem~\ref{main} and the fact that $\sigma_\varphi=\sigma_{\varphi,\mathfrak{u}}$ for
$\mathfrak{u}=(1)_{\alpha\in\Gamma}$.
\end{proof}
%%%%%%%%%%%%%%%%%%%%%%%%%%%%%%%%%%%%
%%%%%%%%%%%%%%%%%%%%%%%%%%%%%%%%%%%%
%%%%%%%%%%%%%%%%%%%%%%%%%%%%%%%%%%%%
%%%%%%%%%%%%%%%%%%%%%%%%%%%%%%%%%%%%
%%%%%%%%%%%%%%%%%%%%%%%%%%%%%%%%%%%%%%%%%%%%%
%%%%%%%%%%%%%%%%%%%%%%%%%%%%%%%%%%%%%%%%%%%%%
%%%%%%%%%%%%%%%%%%%%%%%%%%%%%%%%%%%%
%%%%%%%%%%%%%%%%%%%%%%%%%%%%%%%%%%%%
%%%%%%%%%%%%%%%%%%%%%%%%%%%%%%%%%%%%
%%%%%%%%%%%%%%%%%%%%%%%%%%%%%%%%%%%%
%%%%%%%%%%%%%%%%%%%%%%%%%%%%%%%%%%%%
%%%%%%%%%%%%%%%%%%%%%%%%%%%%%%%%%%%%
%%%%%%%%%%%%%%%%%%%%%%%%%%%%%%%%%%%%
%%%%%%%%%%%%%%%%%%%%%%%%%%%%%%%%%%%%
%%%%%%%%%%%%%%%%%%%%%%%%%%%%%%%%%%%%
%%%%%%%%%%%%%%%%%%%%%%%%%%%%%%%%%%%%
%%%%%%%%%%%%%%%%%%%%%%%%%%%%%%%%%%%%
%%%%%%%%%%%%%%%%%%%%%%%%%%%%%%%%%%%%
%%%%%%%%%%%%%%%%%%%%%%%%%%%%%%%%%%%%
%%%%%%%%%%%%%%%%%%%%%%%%%%%%%%%%%%%%
%%%%%%%%%%%%%%%%%%%%%%%%%%%%%%%%%%%%
%%%%%%%%%%%%%%%%%%%%%%%%%%%%%%%%%%%%
%%%%%%%%%%%%%%%%%%%%%%%%%%%%%%%%%%%%
%%%%%%%%%%%%%%%%%%%%%%%%%%%%%%%%%%%%
%%%%%%%%%%%%%%%%%%%%%%%%%%%%%%%%%%%%

%%%%%%%%%%%%%%%%%%%%%%%%%%%%%%%%%%%%
% \section*{Acknowledgement}
%%%%%%%%%%%%%%%%%%%%%%%%%%%%%%%%%%%%
%\[\underline{\SP\SP\SP\SP\SP\SP\SP\SP\SP\SP\SP\SP\SP\SP\SP\SP}\]

\noindent \noindent{\small {\bf Safoura Arzanesh}, Faculty
of Mathematics, Statistics and Computer Science, College of
Science, University of Tehran, Enghelab Ave., Tehran, Iran
 ({\it e-mail}: arzanesh.parsian@gmail.com)}
\\ \\
{\small {\bf Fatemah Ayatollah Zadeh Shirazi}, Faculty
of Mathematics, Statistics and Computer Science, College of
Science, University of Tehran, Enghelab Ave., Tehran, Iran
\linebreak ({\it e-mail}: f.a.z.shirazi@ut.ac.ir,
fatemah@khayam.ut.ac.ir)}
\\ \\
{\small {\bf Arezoo Hosseini},
Faculty of Mathematics, College of Science, Farhangian University, Pardis Nasibe--shahid sherafat, Enghelab Ave., Tehran, Iran
({\it e-mail}: a.hosseini@cfu.ac.ir)}
\\ \\
{\small {\bf Reza Rezavand}, School of Mathematics, Statistics
and Computer Science, College of Science, University of Tehran,
Enghelab Ave., Tehran, Iran ({\it e-mail}: rezavand@ut.ac.ir)}

\end{document}